\documentclass{amsart}                    % onecolumn
\usepackage{amsmath,amsthm, amsfonts}
\usepackage{graphicx}
\newtheorem{theorem}{Theorem}
\newtheorem{corollary}[theorem]{Corollary}
\newtheorem{remark}[theorem]{Remark}

\def\EE{{\mathcal E}}

\def\LL{{\mathcal L}}

\newcommand\E{{\mathbb E}}
\newcommand\N{{\mathbb N}}
\newcommand\R{{\mathbb R}}

\newcommand\Z{{\mathbb Z}}

\newcommand{\indiq}{\hbox{\rm 1}{\hskip -2.8 pt}\hbox{\rm I}}

\begin{document}

\title[Uniform contractivity in Kac's model]{Uniform contractivity in Wasserstein metric for the original 1D Kac's model
}


\author{Maxime Hauray}


\address{M. Hauray: Institut de Math\'ematiques de Marseille, Universit\'e d'Aix-Marseille, CNRS UMR 7353 et \'Ecole Centrale Mrseille.}

\email{maxime.hauray@univ-amu.fr}




\begin{abstract}
We study here a very popular 1D jump model introduced by Kac: it consists of $N$ velocities encountering random binary collisions at which they randomly exchange energy. We show the uniform (in $N$) exponential contractivity of the dynamics in a non-standard  Monge-Kantorovich-Wasserstein: precisely the MKW metric of order 2 on the energy. The result is optimal in the sense that for each $N$, the contractivity constant is equal to the $L^2$ spectral gap of the generator associated to Kac's dynamic.
As a corollary, we get an uniform but non optimal contractivity in the MKW  metric of order $4$. We use a simple coupling that works better that the parallel one. The estimates are simple and new (to the best of our knowledge).
\end{abstract}

\maketitle

\section{Introduction and results}  \label{sec:intro}

We study here the following model, introduced by Kac in~\cite{Kac1956} and intensively studied since (we refer to the review paper~\cite{CCLReview} or to~\cite{MM} for more details). The velocities $(V_{i,t})_{i \le \N}, t \ge 0 \in \R^N$ of the $N$ particles follows the above dynamics: 
For $N(N-1)/2$ independent Poisson random measures $\bigl(\Pi_{i,j} \bigr)_{i <j}$ on $\R^+ \times \R/ 2\pi \Z$ with the same intensity measure $  \frac1{\pi (N-1) } dt \times d \theta$,
\begin{equation} \label{eq:PS_vel}
V_{i,t} = V_{i,0} +  \sum_{j \neq i} \int_0^t \int_{-\pi}^\pi \Bigl[ \bigl(V_{i,t^-} \cos \theta + V_{j,t^-} \sin \theta \bigr)  - V_{i,t^-} \Bigr] \Pi_{i,j}(dt,d\theta),
\end{equation}
where for $i > j$, the $\Pi_{i,j} = \Pi_{j,i} \circ s$, with $s(\theta) = - \theta$. The normalization in the intensity measure is chosen in order to be consistent with previous work on the subject, in particular~\cite{CCL1}. It means that a particle encounters in average two collisions per unit of time. 
Existence of uniqueness of solutions to that kind of jump process is very standard: for almost all realizations there is at most one collision at each time, \emph{i.e.\ } $\sum_{i < j} \Pi_{i,j} \bigl( \{ t \} \times (-\pi, \pi] \bigr) \le  1$ for all $t \ge 0$, so that the unique solution is easy to construct.

\medskip
A less rigorous but more simple description of the dynamics is the following: at independent random times $\tau_{i,j}$ (for $i \neq j)$ with exponential law $\EE(1)$, the particles $i$ and $j$ exchange their kinetic energy and jump in the following way (with $t = \tau_{i,j}$):
\[
\begin{pmatrix}V_{i,t^-} \\ V_{j,t^-}  \end{pmatrix}
\xrightarrow{collision} 
\begin{pmatrix} V_{i,t} \\ V_{j,t} \end{pmatrix} =
\begin{pmatrix} V_{i,t^-} \cos \theta + V_{j,t^-} \sin \theta \\ 
- V_{i,t^-} \sin \theta + V_{j,t^-} \cos \theta \end{pmatrix},
\] 
where $\theta$ is a r.v. with uniform law on $\R /2\pi \Z$, independent of everything else.

It is possible to use  parallel coupling (same collision and same angle for each collision) between two processes $(V_{i,t})_{i \le \N}$ and $(W_{i,t})_{i \le \N}$ solution of~\eqref{eq:PS_vel}. It leads to a result similar to the one obtained by Tanaka~\cite{Tanaka} for the Boltzmann-Kac model: the coupled evolution is contractive in the MKW metric of order $2$.

However, in 1D there is a better coupling: roughly it is related to the fact that optimal coupling between two uniform measures on (distinct) circles of center $0$ and different radius is obtained when we preserve the angles.
More precisely, we introduce the angle  $\alpha \in (-\pi,\pi]$ such that $(V_{i,t^-}, V_{j,t^-}) = \sqrt{V_{i,t^-}^2 + V_{j,t^-}^2} (\cos \alpha, \sin \alpha)$. Then, the velocities after collision read
\[
(V_{i,t^-},V_{j,t^-}) \xrightarrow{collision} (V_{i,t},V_{j,t})= \sqrt{V_{i,t^-}^2 + V_{j,t^-}^2}
 \bigl(\cos (\alpha- \theta), \sin (\alpha- \theta)  \bigr)
\]
Since the $\theta$ has uniform law and is independent of everything else, so is $\alpha -\theta$. A interesting consequence is that the solution to~\eqref{eq:PS_vel} has exactly the same law than the (unique) solution $(\widetilde V_{i,t})_{i \le N}$ to the process
\begin{equation} \label{eq:PS_vel2}
\widetilde V_{i,t} = V_{i,0} + \sum_{j \neq i} \int_0^t \int_{-\pi}^\pi \Bigl[ \sqrt{\widetilde V_{i,t^-}^2 + \widetilde V_{j,t^-}^2}  \cos \theta -  \widetilde V_{i,t^-} \Bigr] \Pi_{i,j}(dt,d\theta),
\end{equation}
where this time $\Pi_{i,j}$ for $ i >j$ is defined as $\Pi_{i,j} := \Pi_{j,i} \circ r$, with $r(\theta) = \theta +\frac\pi 2$ (in order to transform cosinus in sinus).
Remark that since the angle $\alpha$ introduce above depends continuously on the couple $(V_{i,t^-},V_{j,t^-})$, so that there is no measurability issue when we pass from the first formulation~\eqref{eq:PS_vel} to the second one~\eqref{eq:PS_vel2}.

Since the two systems are equivalent in law, we now focus only on the second one, and drop the tilde in the notation. Using parallel coupling with two systems solutions to~\eqref{eq:PS_vel2}, we obtain an ``optimal'' contractivity result:
\begin{theorem} \label{thm:main}
Assume that $(V_{i,t})_{i \le N}$ and $(W_{i,t})_{i \le N}$ are two solutions to~\eqref{eq:PS_vel2} constructed on the same probability space, with normalized energy : $\sum_i (V_{i,0})^2 = \sum_i (W_{i,0})^2 =N$ a.s..  Then
\[
\E \biggl[ \frac1N \sum_{i =1}^N \bigl[ (V_{i,t})^2 - (W_{i,t})^2  \bigr]^2 \biggr] =
e^{-\lambda_N t}  \, \E \biggl[ \frac1N \sum_{i =1}^N  \bigl(V_{i,0}^2 - W_{i,0}^2  \bigr)^2 \biggr],
\]
with $\lambda_N = \frac12 + \frac3{2(N-1)} = \frac12 \frac{N+2}{N-1} \ge \frac12$.
\end{theorem}

\begin{remark}
Remark that $\lambda_N$ is also the spectral gap in $L^2$ of the associated generator, which was explicitly computed in~\cite{CCL1,Maslen}.  But according to~\cite[Theorem 2.42]{Chen}, the spectral gap is a bound by above on any contraction rate than could be obtain by coupling. So in that sense, the contraction rate $\lambda_N$ obtained above is optimal. 

Moreover, the fact that our coupling cost $(v^2-w^2)^2$ is of order four in the velocity is related to the fact that the eigenfunction associated to the eigenvalue $-\lambda_N$ of the associated generator is a polynomial of degree four: precisely $\sum_i v_i^4 - c$, for some constant $c$. 
\end{remark}

\medskip
This optimal result with a non standard metric (the MKW metric of order 2 on the energy) implies a non optimal result in MKW metric of order $4$:
\begin{corollary} \label{cor}
Under the same assumptions than in Theorem~\ref{thm:main}, we have
\begin{multline*}
\E \biggl[ \frac1N \sum_{i =1}^N \bigl(V_{i,t} - W_{i,t}  \bigr)^4 \biggr]  
\le 
e^{-\lambda_N t}  \, \E \biggl[ \frac1N \sum_{i =1}^N  \bigl( V_{i,0}^2 - W_{i,0}^2  \bigr)^2 \biggr] \\
+ e^{-t}  \, \E \biggl[ \frac1N \sum_{i =1}^N  \bigl( V_{i,0} - W_{i,0} \bigr)^4 \biggr].
\end{multline*}
\end{corollary}

The question whether these results could be extended to the case where the angle $\theta$ are not uniform, or to larger dimension remains open. We refer to the works of Mischler and Mouhot~\cite{MM} and Rousset~\cite{Rousset} for the best results in higher dimension.

In order to obtain propagation of chaos towards the limit nonlinear jump process, the strategy used in~\cite{Fontbona} could be applied.

\medskip
\paragraph{ \bf Extension to the nonlinear limit model}

A similar result could be obtained on the limit nonlinear jump process
\begin{equation} \label{eq:lim}
V_t = V_0 +  \int_0^t \int_{-\pi}^\pi \Bigl[ \bigl(V_{t^-} \cos \theta + F^{-1}_{\LL( V_{t^-})}(u) \sin \theta \bigr)  - V_{t^-} \Bigr] \Pi(dt,d\theta,du),
\end{equation} 
where $\Pi$ is a Poisson random measure on $\R^+ \times \R/ 2\pi \Z \times [0,1]$ with intensity measure $\frac1{\pi (N-1) } dt \times d \theta \times du$, and $F^{-1}_{\LL( V_{t^-})}$ stands for the pseudo-inverse of the cumulative distribution function of $V_{t^-}$.

Similarly, that process is equivalent (in law) to
\begin{equation} \label{eq:lim2}
V_t = V_0 +  \int_0^t \int_{-\pi}^\pi \biggl[ \sqrt{ V_{t^-}^2 + \bigl(F^{-1}_{\LL( V_{t^-})}(u)\bigr)^2} \cos \theta - V_{t^-} \biggr] \Pi(dt,d\theta,du),
\end{equation} 
and we have a similar result of contraction:

\begin{theorem}
Assume that $(V_t)_{t \ge 0}$ and $(W_t)_{t \ge 0}$ are two solutions of~\eqref{eq:lim2} constructed with the same PRM $\Pi$. Then,
\[
\E \bigl[ (V_t^2 - W_t^2)^2 \bigr] = e^{-\frac t 2} \E \bigl[ (V_0^2 - W_0^2)^2 \bigr].
\]
\end{theorem}

We will not prove that theorem, since the proof is really similar to the one of Theorem~\ref{thm:main}.
A similar corollary on the decrease of the order four MKW distance could be stated. 

%%%%%%%%%%%%%%%%%%%%%%%%%%%%%%%%%%%%%%
\section{Proofs}

\begin{proof}[Proof of Theorem~\ref{thm:main}]
In order to simplify the writing, we introduce notations for the energy (up to a factor $2$): $E_{i,t} := (V_{i,t})^2$ and 
$F_{i,t} := (W_{i,t})^2$.  Then, both energy vectors satisfy the following stochastic process:
\begin{equation} \label{eq:PS_ener}
E_{i,t} = E_{i,0} + \sum_{j \neq i} \int_0^t \int_{-\pi}^\pi \bigl[ \bigl(E_{i,t^-} + E_{j,t^-} \bigr)  \cos^2 \theta -   E_{i,t^-} \bigr] \Pi_{i,j}(dt,d\theta),
\end{equation}
where $\Pi_{i,j} := \Pi_{j,i} \circ r$ for $ i>j$, with $r(\theta) = \theta +\frac\pi 2$, as before. The $(F_{i,t})_{i \le N, t \ge 0}$ satisfies the same equations, and so does $G_{i,t} := E_{i,t} - F_{i,t}$ by linearity of~\eqref{eq:PS_ener}.

Next, since the total energy is preserved by the dynamics, we have almost surely for all $t>0$:
\begin{equation} \label{cons_ener}
\frac1N \sum_{i=1}^N  E_{i,t} = \frac1N \sum_{i=1}^N  F_{i,t}=1 , \qquad 
\frac1N \sum_{i=1}^N  G_{i,t} = 0.
\end{equation}
We are interested in the evolution of $\frac1N \sum_i G_{i,t}^2$. But the process $\bigl(G_{i,t}^2\bigr)_{i \le N, t\ge 0}$ satisfies
\[
G_{i,t}^2 = G_{i,0}^2 + \sum_{j \neq i} \int_0^t \int_{-\pi}^\pi \bigl[ (G_{i,t^-} + G_{j,t^-})^2  \cos^4 \theta -   G_{i,t^-}^2 \bigr] \Pi_{i,j}(dt,d\theta).
\]
It implies that
\[
\frac d{dt} \E \Bigl[ \frac1N \sum_i G_{i,t}^2  \Bigr] =   \E \biggl[ \frac2{N(N-1)} \sum_{i < j} \Bigl( (G_{i,t}+ G_{j,t})^2 (\cos^4 \theta + \sin^4 \theta) - G_{i,t}^2  - G_{j,t})^2  \Bigr) \biggr].
\]
Since $\frac1{2\pi}\int_{-\pi}^\pi (\cos^4 \theta + \sin^4 \theta) \, d \theta = \frac34$, it also reads
\[
\frac d{dt} \E \Bigl[ \frac1N \sum_i (G_{i,t})^2  \Bigr] =   - \frac12 \E \Bigl[ \frac1N \sum_i (G_{i,t})^2  \Bigr]
+   \frac3{2N(N-1)} \E \Bigl[ \sum_{i \neq j} G_{i,t} G_{j,t} \Bigr].
\]
We emphasize that we have replaced the sum on $\{i <j\}$ by a sum on $\{i \neq g \}$. But 
\[
\sum_{i \neq j} G_{i,t} G_{j,t} =  \Bigl( \sum_i G_{i,t} \Bigr)^2 - \sum_i (G_{i,t})^2 =  - \sum_i (G_{i,t})^2,
\]
thanks to the conservation of energy~\eqref{cons_ener}. So we finally have
\[
\frac d{dt} \E \Bigl[ \frac1N \sum_i (G_{i,t})^2  \Bigr] =   - \Bigl( \frac12 + \frac3{2(N-1)}  \Bigr)  \E \Bigl[ \frac1N \sum_i (G_{i,t})^2  \Bigr], 
\]
and the conclusion follows.
\end{proof}

\begin{proof}[Proof of Corollary~\ref{cor}]
In equation~\eqref{eq:PS_vel2}, we see that the sign of $V_{i,t}$ after a collision is the one of $\cos \theta$, where $\theta$ is the angle used in that collision,  (and the sign of $V_{j,t}$ is the one of $\sin \theta$). Since the same angle are used during the evolution of $(W_{i,t})_{i \le N, t \ge 0}$, it implies than $V_{i,t}$ and $W_{i,t}$ have the same sign after the first collision involving the $i$-th particle. And if they have the same sign, 
\[
 (V_{i,t} -  W_{i,t})^4 \le (V_{i,t} -  W_{i,t})^2 (V_{i,t} +   W_{i,t})^2 = \bigl(V_{i,t}^2 -  W_{i,t}^2 \bigr)^2.
\]
We call $\tau_i$ the stopping time defined has the time of the first collision involving $i$. We get
\begin{align*}
\E \biggl[ \frac1N \sum_{i =1}^N   \bigl(V_{i,t} - W_{i,t}  \bigr)^4 \biggr]   & \le
\E \biggl[ \frac1N \sum_{i =1}^N   \bigl(V_{i,t}^2 - W_{i,t}^2  \bigr)^2   \indiq_{ t \ge \tau_i} \biggr]  \\
 & \hspace{60pt} +   \E \biggl[ \frac1N \sum_{i =1}^N   \bigl(V_{i,t} - W_{i,t}  \bigr)^4    \indiq_{ t \le \tau_i} \biggr]  \\
% & \le  \E \biggl[ \frac1N \sum_{i =1}^N   \bigl(V_{i,t}^2 - W_{i,t}^2  \bigr)^2  \biggr]  \\
% & \hspace{60pt} +  \frac1N \sum_{i =1}^N  \E \biggl[    \bigl(V_{i,0} - W_{i,0}  \bigr)^4    \indiq_{ t \le \tau_i} \biggr]  \\
\end{align*}
The first term in the r.h.s. is bounded thanks to the previous result (the indicator function is bounded by $1$).
To bound the second term, remark that each $\tau_i$ is independent of the initial condition and as exponential law $\EE(1)$, so that
\begin{align*}
\E \biggl[  & \frac1N \sum_{i =1}^N   \bigl(V_{i,t} - W_{i,t}  \bigr)^4    \indiq_{ t \le \tau_i} \biggr] 
= \frac1N \sum_{i =1}^N  \E \biggl[    \bigl(V_{i,0} - W_{i,0}  \bigr)^4    \indiq_{ t \le \tau_i} \biggr] \\
& = \frac1N \sum_{i =1}^N  \E \bigl[    \bigl(V_{i,0} - W_{i,0}  \bigr)^4     \bigr] \mathbb P \bigl(   \tau_i \ge t \bigr)
=  e^{-t}  \E \biggl[  \frac1N \sum_{i =1}^N \bigl(V_{i,0} - W_{i,0}  \bigr)^4     \biggr].
\end{align*}
The conclusion follows.
\end{proof}

{\bf Acknowledgements}
The authors would like to thank Matthias Rousset and Alexandre Gaudilli\`ere for very stimulating discussion about that problem.
In particular the proof of Corollary~\ref{cor} is due to the latter.



\end{document}